\documentclass[12pt]{amsart}
\usepackage[latin1]{inputenc}
\usepackage{amssymb, amsmath, amscd, mathrsfs,marvosym, psfrag,color,a4} 

\input xy
\xyoption{all}
\DeclareMathAlphabet{\mathpzc}{OT1}{pzc}{m}{it}

\newtheorem{theorem}{Theorem}[section]

\newtheorem{proposition}[theorem]{Proposition}

\newtheorem{lemma}[theorem]{Lemma}

\theoremstyle{definition}
\newtheorem{definition}[theorem]{Definition}

\theoremstyle{remark}
\newtheorem{remark}[theorem]{Remark}
\newtheorem{remarks}[theorem]{Remarks}

\def\le{\leqslant}

\newcommand{\CA}{{\mathcal A}}

\newcommand{\CC}{{\mathcal C}}

\newcommand{\CE}{{\mathcal E}}

\newcommand{\CG}{{\mathcal G}}
\newcommand{\CH}{{\mathcal H}}
\newcommand{\CI}{{\mathcal I}}
\newcommand{\CJ}{{\mathcal J}}

\newcommand{\CO}{{\mathcal O}}

\newcommand{\CU}{{\mathcal U}}
\newcommand{\CV}{{\mathcal V}}
\newcommand{\CW}{{\mathcal W}}

\newcommand{\CZ}{{\mathcal Z}}

\newcommand{\lDelta}{{\ol \Delta}}

\newcommand{\SA}{{\mathscr A}}
\newcommand{\SB}{{\mathscr B}}
\newcommand{\SC}{{\mathscr C}}
\newcommand{\SD}{{\mathscr D}}

\newcommand{\SM}{{\mathscr M}}
\newcommand{\SN}{{\mathscr N}}

\newcommand{\fg}{{{\mathfrak g}}}

\newcommand{\hCW}{{\widehat\CW}}

\newcommand{\lCG}{{\ol\CG}}

\newcommand{\lCZ}{{\ol\CZ}}

\newcommand{\lCV}{{\ol\CV}}

\newcommand{\tCG}{\widetilde{\CG}}

\newcommand{\DR}{{\mathbb R}}
\newcommand{\DZ}{{\mathbb Z}}

\newcommand{\DN}{{\mathbb N}}

\newcommand{\bM}{{\mathbf M}}

\newcommand{\bH}{{\mathbf H}}

\newcommand{\hbH}{\widehat{\mathbf H}}

\newcommand{\ch}{{\operatorname{char}\, }}

\newcommand{\End}{{\operatorname{End}}}

\newcommand{\Hom}{{\operatorname{Hom}}}

\newcommand{\supp}{{\operatorname{supp}}}
\newcommand{\catmod}{{\operatorname{-mod}}}

\newcommand{\im}{{\operatorname{im}\,}}

\newcommand{\ol}{\overline}

\newcommand{\id}{{\operatorname{id}}}

\newcommand{\sur}{\mbox{$\to\!\!\!\!\!\to$}}
\newcommand{\linie}{{\,\text{---\!\!\!---}\,}}
\newcommand{\llinie}{{\text{---\!\!\!---\!\!\!---}}}

\newcommand{\Aff}{{\operatorname{Aff}}}

\newcommand{\comment}[1]{}

\newcommand{\refl}{{\operatorname{ref}}}
\newcommand{\fin}{{\operatorname{fin}}}

\begin{document}

\pagenumbering{arabic}
\title[]{Filtered  moment graph sheaves} \author[]{Peter Fiebig, Martina Lanini}
\begin{abstract} We introduce the notion of (co-)filtered sheaves on quotients of moment graphs by a group action. We then introduce a (co-)filtered version of the canonical sheaves of Braden and MacPherson and show that their global sections are the indecomposable projective objects in a suitably defined exact category. 
\end{abstract}

\address{Department Mathematik, FAU Erlangen--N\"urnberg, Cauerstra\ss e 11, 91058 Erlangen}
\email{fiebig@math.fau.de, lanini@math.fau.de}
\maketitle
\section{Introduction}

Moment graphs provide a  powerful tool for the study of character and multiplicity problems in representation theory. 
For example, the Kazhdan-Lusztig conjecture on characters of irreducible highest weight representations of a finite dimensional complex Lie algebra $\fg$ can be reformulated in terms of sheaves on moment graphs. This is based on an equivalence between certain exact categories of ``objects admitting a Verma flag'' (cf. \cite{FieAdv}). 

Here are the main ideas for this equivalence. The BGGH-reciprocity result translates the character problem of Kazhdan and Lusztig into the problem of calculating the Verma multiplicities of indecomposable projective objects in the BGG-category $\CO$ associated with $\fg$: these multiplicities should be given by evaluating Kazhdan--Lusztig polynomials at $1$. The Kazhdan--Lusztig polynomials are the entries of the base change matrix with respect to the standard and self-dual $\DZ[v^{\pm 1}]$-bases of the Hecke algebra $\bH$ of $\fg$. 

Now, to any block $\CO_\Lambda$ of $\CO$ one can associate a moment graph $\CG=\CG_\Lambda$. We denote by $\CZ$ the structure algebra of $\CG$ and let $\CC^{\fin}$ be the $\DZ$-graded category of   $\CZ$-modules {\em admitting a Verma flag}. The Grothendieck group of $\CC^{\fin}$ can naturally by identified with the $\DZ[v^{\pm 1}]$-module underlying the  Hecke algebra $\bH$. Even though $\CC^{\fin}$ is not an abelian category, it carries a non-standard exact structure in such a way that it is equivalent, as an exact category, to the category of modules admitting a Verma flag in a deformed version of $\CO_\Lambda$. In particular, the projective objects in $\CC^{\fin}$ correspond to the projectives in the deformed version of $\CO_\Lambda$. The Kazhdan--Lusztig conjecture is now equivalent to the statement that   in the Grothendieck group of $\CC^{\fin}$ the indecomposable projective objects correspond to the Kazhdan--Lusztig self-dual elements. 

The advantage of the above reformulation is that there is an algorithm of linear-algebraic nature that constructs the indecomposable projective objects in $\CC^{\fin}$.  In order to give a combinatorial model for the equivariant intersection cohomology of complex varieties with a torus action, Braden and MacPherson described in \cite{BMP}
 an algorithm that produces the  {\em canonical sheaves} on a moment graph. It turns out that $\CG_\Lambda$ is the moment graph of a (partial) flag variety associated with the Langlands dual Lie algebra and that  the global sections of the Braden--MacPherson sheaves yield the indecomposable projectives in $\CC^{\fin}$. Hence these sheaves encode character formulas for the simple highest weight modules for complex semisimple Lie algebras.
 
 We would like to establish an analogous result in a case where character formulas are not known in general, not even conjecturally. Let $G$ be a reductive group defined over an algebraically closed field $k$ of characteristic $p>0$. Lusztig conjectured in \cite{LusProc} a  formula for the characters of restricted irreducible representations of $G$. This formula is known to be true if  the characteristic $p$ is large enough. More precisely, for any root system $R$ there is a number $N$ such that the formula holds for all groups $G$ with root system $R$ that are defined over a field $k$ of characteristic $p>N$.  Many representation theorists were hoping, for a long time, that the Coxeter number of $R$ could be taken for $N$, but is not true (cf. \cite{W}). An upper bound for $N$ can be found in \cite{FieCrelle}. 
 
 However, if $p$ is at least the Coxeter number, one can obtain partial information from an affine version of the moment graph theory used in the complex case:  In \cite{FieJAMS} a functor from affine moment graph sheaves to restricted modules over the Lie algebra of $G$ (graded by the weight lattice) is constructed. This functor was used to establish an upper bound on the relevant Jordan-H\"older multiplicities. However, it is not clear whether this upper bound is in fact sharp and the functor itself does not establish a categorical equivalence.  It  is hence natural to ask whether it is possible to find a more suitable moment graph theory for the modular setting that first gives precise information on the multiplicities and secondly yields a categorical equivalence. This is the problem we want to investigate in this paper, in \cite{FieLan} and in further forthcoming papers.

In the modular setting outlined above, the role of the Hecke algebra is played by  the {\em periodic module} of the affine Hecke algebra (cf. \cite{LusAdv}). To $G$ we associate the set of alcoves $\SA$ that is acted upon by the affine Weyl group $\hCW$. The set $\SA$ carries a certain partial order, the {\em generic Bruhat order}, and a quantization of this data yields the periodic module $\bM$ over the affine Hecke algebra $\hbH$. Again the periodic polynomials  appear as entries in a base change $\SA\times\SA$-matrix associated with a standard and a self dual basis (of a submodule of $\bM$).  The name {\em periodic polynomials} stems from the fact that these entries are invariant  under simultaneous translation of row and column indices by an element of the root lattice. 

This is our motivation to consider the following: We associate with the $\hCW$-action on $\SA$ a {\em periodic moment graph} $\tCG$. It is acted upon by the root lattice $\DZ R$. We then consider the quotient graph $\lCG$ and introduce a category of sheaves on $\lCG$ that carry a cofiltration indexed by poset ideals in $\SA$. Then we introduce a global section functor that associates with such a sheaf a similarly cofiltered $\lCZ$-module.  Inside the category of cofiltered $\lCZ$-modules  we define the subcategory $\CC$ of objects admitting a Verma flag, and the (larger) category $\CC^{\refl}$ of reflexive objects. We endow both categories with an exact structure. Finally we describe a cofiltered version of the Braden--MacPherson algorithm and show that the global sections of the resulting sheaves yield the indecomposable projectives in $\CC^{\refl}$. 

We do not know how to show that these global sections actually admit a Verma flag, i.e. are contained in $\CC$, and this statement might even be false in general situations. However, in the companion paper \cite{FieLan} we study $\CC^{\refl}$ and the projectives therein from a very different perspective. The translation functors  (associated with the simple affine reflections in $\hCW$) yield an alternative and independent construction of the projectives in $\CC^{\refl}$. This construction shows, in addition, that the projectives in $\CC^{\refl}$ do actually admit a Verma flag, i.e. are contained in  $\CC$. Moreover, in \cite{FieLan} we prove the result that justifies our effort: For almost all characteristics (this includes characteristic $0$), the stalk multiplicities of the indecomposable projectives are indeed given by Lusztig's periodic polynomials.
  
We will discuss the representation theoretic applications of our theory of cofiltered sheaves in a forthcoming paper. 
In particular, we will relate our category with the Andersen-Jantzen-Soergel combinatorial category from \cite{AJS}.

\section{Acknowledgements}
This work was partially supported by the DFG grant SP1388.

\section{Moment graphs with a group action}

In this section we collect some basic notions and results. 
\subsection{Moment graphs}
For the purposes of this article, the following is a convenient definition of a moment graph. It differs slightly from previous definitions  (see, e.g., \cite{FieAdv}), but the difference is not significant at all. 
Let $X$ be a lattice, i.e.~a free abelian group of finite rank.
\begin{definition} A moment graph over $X$ is a graph given by a set of vertices $\CV$, a set of edges $\CE$, and a labeling $\alpha\colon \CE\to X/\{\pm1\}$. 
We assume that $\CG$ has no loops, i.e. no vertex is connected to itself by an edge. 
\end{definition}
We use the notation $x\stackrel{\alpha}\linie y$ for an edge connecting $x,y\in\CV$ with label $\alpha\in X/\{\pm 1\}$.
\begin{definition}
An {\em ordered moment graph} is a moment graph as above, together with a partial order $\le$ on the set $\CV$ such that $x,y\in\CV$ are comparable if there is an edge $x\llinie y$. 
\end{definition}

\subsection{Moment graphs with a group action}\label{subsec-groupact}

Morphisms of moment graphs were introduced in \cite{LanJoA}. We won't need the most general notion of a morphism, so the following will suffice: If $\CG_1$ and $\CG_2$ are moment graphs over $X$, then we call an isomorphism $\sigma\colon \CG_1\to\CG_2$ of the underlying abstract graphs a {\em moment graph isomorphism} if it preserves the labels, i.e. if $\sigma$ maps the edge $x\stackrel{\pm\alpha}\linie y$ to the edge $\sigma(x)\stackrel{\pm\alpha}\linie\sigma(y)$. 
Now let $G$ be a group. 
\begin{definition} A {\em  $G$-moment graph} is a moment graph $\CG$ together with an action of $G$ on $\CG$ by moment graph automorphisms.
\end{definition}

Let $\CG$ be a  $G$-moment graph. 
\begin{definition} \label{def-GMom}The quotient of $\CG$ by the $G$-action is the moment graph $\lCG=G\backslash \CG$ that is given by the following data:
\begin{enumerate} 
\item Its set of vertices is $\ol \CV=G\backslash \CV$, the set of  $G$-orbits in $\CV$.
\item Let $\Omega, \Omega^\prime\in\ol\CV$, $\Omega\ne\Omega^\prime$ be vertices and $\pm\alpha\in X/\langle \pm 1\rangle$.  Then  $\Omega$ is connected to $\Omega^\prime$ by a $\pm\alpha$-edge in $\lCG$ if and only if there are vertices $x\in \Omega$ and $x^\prime\in\Omega^\prime$ of $\CG$ that are connected by a $\pm\alpha$-edge.
\end{enumerate}
\end{definition}

\begin{remarks} Note that a vertex $v\in\Omega$ can be connected to several  $v^\prime\in\Omega^\prime$ by a $\pm\alpha$-edge. Even if this is the case, we only have a single $\pm\alpha$-edge between $\Omega$ and $\Omega^\prime$. \end{remarks}

In the next section we introduce the (for us) motivating example of a moment graph with a group action.
\subsection{An example: the periodic moment graph}

From now on we fix  a finite irreducible root system $R$ in a real finite dimensional vector space  $V$.  For any $\alpha\in R$ we denote by $\alpha^\vee\in V^\ast=\Hom_\DR(V,\DR)$ the corresponding coroot. 
The weight and coweight lattices are 
\begin{align*}
X&=\{\lambda\in V\mid \langle \lambda,\alpha^\vee\rangle\in\DZ\text{ for all $\alpha^\vee\in R^\vee$}\},\\
X^\vee&=\{v\in V^\ast\mid \langle \alpha,v\rangle\in\DZ\text{ for all $\alpha\in R$}\}.
\end{align*}
Here, and in the following, we denote by $\langle\cdot,\cdot\rangle\colon V\times V^\ast\to \DR$ the natural pairing. We have $\DZ R\subset X$.

For any $\alpha^\vee\in R^\vee$ and $n\in\DZ$ we define
\begin{align*}
 H_{\alpha^\vee,n}&:=\{\mu\in V\mid \langle \mu,\alpha^\vee\rangle = n\}, \\
H_{\alpha^\vee,n}^+&:=\{\mu\in V\mid \langle \mu, \alpha^\vee\rangle>n\},\\
H_{\alpha^\vee,n}^-&:=\{\mu\in V\mid \langle \mu,\alpha^\vee\rangle<n\}.
\end{align*}
The $H_{\alpha^\vee,n}$ are called {\em reflection hyperplanes}. We clearly have $H^{(\pm)}_{\alpha^\vee,n}=H^{(\mp)}_{-\alpha^\vee,-n}$.
 The connected components of $V\setminus\bigcup_{\alpha^\vee\in R^{\vee,+},n\in\DZ}H_{\alpha^\vee,n}$ are called {\em alcoves}. We denote by $\SA$ the set of alcoves.  

For $\lambda\in X$ we denote by $t_\lambda\colon V\to V$ the affine translation $\mu\mapsto \lambda+\mu$. We clearly have
\begin{align*}
t_\lambda(H_{\alpha^\vee,n})=H_{\alpha^\vee,n+\langle\lambda,\alpha^\vee\rangle},\quad
t_\lambda(H_{\alpha^\vee,n}^\pm)=H_{\alpha^\vee,n+\langle\lambda,\alpha^\vee\rangle}^\pm.
\end{align*}
Hence $t_\lambda$ induces a bijection $t_{\lambda,\SA}\colon \SA\to\SA$.

We  denote by
\begin{align*}
s_{\alpha^\vee,n}\colon V&\to V\\
\lambda&\mapsto \lambda-(\langle \lambda,\alpha^\vee\rangle-n)\alpha
\end{align*}
the affine reflection with fixed point hyperplane $H_{\alpha^\vee,n}$. Note that each reflection $s_{\alpha^\vee,n}$ stabilizes the weight lattice $X$ and the root lattice $\DZ R$.  The {\em affine Weyl group} is the subgroup $\hCW\subset\Aff(V)$ generated by the affine reflections $s_{\alpha^\vee,n}$ for $\alpha^\vee\in R^\vee$ and $n\in\DZ$.  
The finite Weyl group $\CW$ is the subgroup of $\hCW$ generated by the linear reflections $s_{\alpha^\vee}=s_{\alpha^\vee,0}$. The root lattice $\DZ R\subset X$ is  contained  in $\hCW$, as
$$
s_{\alpha^\vee,1}s_{\alpha^\vee,0}=t_\alpha
$$
for each $\alpha^\vee\in R^\vee$. It is  a normal subgroup and we have
$$
\hCW=\CW\ltimes \DZ R.
$$

We denote by $\preceq$ the partial order on $\SA$ that is generated by the relations $A\preceq s_{\alpha^\vee,n}(A)$ if $A\subset H_{\alpha^\vee,n}^-$ (and hence $s_{\alpha^\vee,n}(A)\in H_{\alpha^\vee,n}^+$) for a positive coroot $\alpha^\vee$ and $n\in\DZ$. 
Note that the translations $t_{\lambda,\SA}\colon \SA\to\SA$ preserve the order $\preceq$, i.e.  we have 
$$
A\preceq B \text{ if and only if } t_\gamma(A)\preceq t_\gamma(B) \text{ for all $\gamma\in X$}.
$$
This order is called the {\em generic Bruhat order}

Now we define the for us most important example of a moment graph with a group action. 
\begin{definition} The moment graph $\tCG$ is the following moment graph over the lattice $X^\vee$:
\begin{itemize}
\item Its set of vertices is the set of alcoves $\SA$. This set is partially ordered by the generic Bruhat order $\preceq$.
\item The vertices  $A,B\in\SA$ are connected by an edge if there is $\alpha\in R^+$ and $n\in\DZ$ with $B=s_{\alpha^\vee,n}A$. 
\item The edge $A\linie s_{\alpha^\vee,n}A$ is labeled by $\pm\alpha^\vee\in X^\vee/\langle\pm 1\rangle$.
\end{itemize}
\end{definition}

\begin{lemma} For $\lambda\in\DZ R$, the bijection $t_{\lambda,\SA}\colon\SA\to\SA$, $A\mapsto \lambda+A$ induces an automorphism of the moment graph $\tCG$. 
\end{lemma}
\begin{proof} We view $t_\lambda$ as a bijection on the set of vertices of $\tCG$. So we only have to show that $t_\lambda$ preserves edges and labels. We have $t_\lambda\circ s_{\alpha^\vee,n}\circ t_{-\lambda}=s_{\alpha^\vee,n+\langle\lambda,\alpha^\vee\rangle}$, hence an edge $A\stackrel{\pm\alpha^\vee}\llinie s_{\alpha^\vee,n}(A)$ is mapped to the edge $t_\lambda(A)\stackrel{\pm\alpha^\vee}\llinie s_{\alpha^\vee,n+\langle\lambda,\alpha^\vee\rangle}(t_\lambda(A))$.  
\end{proof}
Hence we obtain an action of the group $\DZ R$ on $\tCG$.
 Here is the description of the corresponding quotient graph. 
\begin{lemma} The quotient graph $\lCG=\DZ R\backslash \tCG$ can be described as follows:
\begin{itemize}
\item Its set of vertices is the set underlying the finite Weyl group $\CW$.
\item The vertices $x,y\in\CW$ are connected by an edge if there is $\alpha^\vee\in R^\vee$ such that $x=s_{\alpha^\vee}y$. This edge is labelled by $\pm\alpha^\vee$.
\end{itemize}
\end{lemma}
\begin{proof} This immediately follows from the fact that $\SA$ is a principal homogeneous set for the $\hCW$-action and that $\hCW=\CW\ltimes\DZ R$.    
\end{proof}
Before we study moment graphs with a group action in more detail, we collect some more definitions and facts for arbitrary moment graphs.  
\subsection{The structure algebra}
Let $\CH$ be a  moment graph over the lattice $X$ with vertex set $\CV$, and let $k$ be a field. 
We  denote by $X_k=X\otimes_\DZ k$ the $k$-vector space associated with $X$. For $\lambda\in X$ we write $\lambda$ instead of $\lambda\otimes 1\in X_k$. 
Let $S=S(X_k)$ be the symmetric algebra of the $k$-vector space $X_k$. We consider $S$ as a $\DZ$-graded algebra with $X_k\subset S$ being the homogeneous component of degree 2. Most of the modules we deal with in this paper are $\DZ$-graded and we denote the graded components by $M=\bigoplus_{n\in\DZ} M_n$. We assume that all homomorphisms between graded modules respect the grading. 

\begin{definition} The {\em structure algebra} of $\CH$  over the field $k$ is
$$
\CZ=\CZ_k(\CH)=\left\{(z_x)\in\prod_{x\in\CH} S\left|
\begin{matrix}
\text{ $z_x\equiv z_y\mod\alpha(E)$ for all edges }\\
\text{ $E\colon x\llinie y$ of $\CH$}
\end{matrix}\right\}\right..
$$
\end{definition}
This is an algebra by componentwise addition and multiplication. It contains a copy of $S$ on the diagonal, and it is a unital, associative, graded $S$-algebra (note that we take the product above in the category of graded $S$-modules, i.e. componentwise). 

\subsection{Localizations}
Let $L$ be a sublattice of $X$. We say that $L$ is {\em saturated} if the quotient $X/L$ is torsion free, i.e. if the inclusion $L\subset X$ splits. Let $L$ be a saturated sublattice of $X$.

\begin{definition} We denote by $\CH^L$ the sub-moment graph of $\CH$ that we obtain by keeping all vertices and by deleting all edges that are labeled by elements that are not contained in $L/\langle\pm 1\rangle\subset X/\langle \pm1 \rangle$.
We denote by $C(\CH^L)$ the set of connected components of $\CH$ (as an abstract graph).
\end{definition} 
For example,  we can identify $C(\CH^0)$ with the set of vertices $\CV$.

We set    
$$
S^{L}:=S[\alpha^{-1}\mid \alpha\in X, \alpha\ne 0, \alpha\not\in L].
$$
As we invert homogeneous elements, this is a graded algebra as well. 
We have, in particular,  $S^X=S$ and $S^{0}=S[\alpha^{-1}\mid \alpha\in X,\alpha\ne 0]$. For $L\subset L^\prime$
we have a canonical injective homomorphism $S^{L^\prime}\subset S^L$. For an $S$-module $M$ we set 
$$
M^{L}:=M\otimes_S S^{L}.
$$
If $M$ is a graded module, then so is $M^L$ in a natural way.  But note that since we invert homogeneous elements in $S$ of degree 2, we  have an isomorphism $M^L\cong M^L[2]$ if $L\ne X$.

\begin{lemma}\label{lemma-locstruc} Suppose that $\CH$ is a finite moment graph. 
Then the canonical inclusion $\iota\colon\CZ(\CH)\to\CZ(\CH^L)$ is a bijection after localization, i.e. $\iota^L\colon \CZ(\CH)^L\to\CZ(\CH^L)^L$ is an isomorphism.
\end{lemma}

\begin{proof} As $\iota$ is injective and $\CZ(\CH^L)$ is torsion free as an $S$-module, the map $\iota^L$ is injective.  Let $\alpha_1$, \dots, $\alpha_n$ be the labels of all edges of $\CH$ that are deleted in $\CH^L$. Then $\alpha_1\cdots\alpha_n\CZ(\CH^L)\subset\CZ(\CH)$, which implies the surjectivity of $\iota^L$.
\end{proof}

\subsection{Generic decompositions}\label{subsec-gendec}

Let $M$ be a module over  $\CZ$.
\begin{definition}\begin{enumerate}
\item  For a vertex $x$ of $\CH$ we define $M^{0,x}\subset M^0$ as the sub-$\CZ$-module of all elements $m\in M^{0}$ on which $(z_w)\in\CZ$ acts as multiplication by $z_x\in S$.
\item We say that $M$ is {\em generically semisimple} if $M^{0}=\bigoplus_{x\in\CV} M^{0,x}$. 
\end{enumerate}
\end{definition}
If $\CH$ is an infinite graph and $M=\CZ$ as a $\CZ$-module it can occur that $M^{0,x}=0$ for all $x$, so $M$ is not generically semisimple. On a finite graph this cannot happen, as we will show next.

\begin{lemma} Suppose that $\CH$ is a finite moment graph. Then every $\CZ$-module  $M$ is generically semisimple.
\end{lemma}
\begin{proof} Note that $M^0$ is a $\CZ^0$-module. By Lemma \ref{lemma-locstruc} we have $\CZ^0=\bigoplus_{x\in\CV} S^0$, which implies generic semisimplicity.
\end{proof}

\begin{definition}\label{def-suppZmod}  The support of the $\CZ$-module $M$ is
$$
\supp_{\lCZ}\, M=\{x\in\CV\mid M^{0,x}\ne 0\}.
$$
\end{definition}
\subsection{Canonical decompositions}\label{subsec-candec}

We introduce here a slight generalization of the generic decomposition of the previous section. 
 \begin{definition} Let $L\subset X$ be a saturated subset.
We denote by $\CZ^L\catmod^{f}$ the full subcategory of the category of all $\CZ^L=\CZ(\CH)^L$-modules that contains all objects $M$ that are torsion free as $S$-modules and finitely generated as $S$-modules. 
\end{definition}
For saturated subsets $L\subset L^\prime$ of $X$, localization $M\mapsto M^L=M\otimes_{S} S^L$ yields a functor $\CZ^{L^\prime}\catmod^f\to\CZ^L\catmod^{f}$. We will write $\CZ\catmod^f$ instead of $\CZ^X\catmod^f$.

\begin{lemma}\label{lemma-candec} Suppose that $M$ is an object in $\CZ\catmod^f$. Let $L$ be a saturated subset of $X$. Then $M^L$ carries a canonical action of $\CZ(\CH^L)$. In particular, we have a canonical decomposition $M^L=\bigoplus_{\Omega\in C(\CH^L)} M^{L,\Omega}$ of $\CZ(\CH^L)$-modules such that $M^{L,\Omega}\subset\bigoplus_{x\in\Omega}M^{0,x}$.
\end{lemma}
We call the decomposition above the {\em canonical decomposition} of $M^L$.
\begin{proof} As we assume that $M$ is  torsion free as an $S$-module, we have canonical inclusions $M\subset M^L\subset M^0=\bigoplus_{x\in\CV}M^{0,x}$. Hence we can view each $m\in M$ as an $\CV$-tuple $m=(m_x)$ in such a way that $z.m=(z_xm_x)$ for each $z=(z_x)\in\CZ$. The same formula defines an action of $\CZ^L$ on $M^L$.
\end{proof}

\section{Filtered $\lCZ$-modules}

In this section we will define and study a category that is naturally associated with an ordered moment graph  with a group action. So let $G$ be a group acting on the moment graph $\CG$ by moment graph automorphisms. Suppose furthermore that $\CG$ is an ordered moment graph with vertex set $\CV$. We denote the partial order by $\le$. Note that we do not care whether the $G$-action preserves the order or not. As before we will denote by $\lCG=G\backslash \CG$ the quotient graph and by $\lCV=G\backslash \CV$ its set of vertices.

We now fix a field $k$ and denote by $\lCZ=\CZ(\lCG)$ the  structure algebra of $\lCG$ over $k$. For the rest of this paper we assume that $\lCG$ has the following properties: 
\begin{enumerate}
\item The graph $\lCG$ is finite and contains no double edges. 
\item The pair $(\lCG,k)$ satisfies the GKM-condition, i.e. we have $\ch k\ne 2$ and  for each pair of distinct edges $E$ and $E^\prime$ which share a common vertex we have $\alpha(E)\not\in k \alpha(E^\prime)$. 
\end{enumerate}

Before we introduce the category of cofiltered sheaves on $\lCG$, we  collect some general abstract notions and results. 
\subsection{A topology on partially ordered sets}\label{subsec-TopPar}
Note that we have a partial order $\le$ on the set $\CV$. 
For an element $x$ of $\CV$ we will use the short hand notation $\{\le x\}=\{y\in\CV\mid y\le x\}$. The notations $\{\ge x\}$, $\{< x\}$, etc. have an analogous meaning. We will consider the following topology on $\CV$.
\begin{definition} A subset $\CU$ of $\CV$ is called {\em open}, if $x\in \CU$ and $y\le x$ implies $y\in \CU$, i.e. if 
$$
\CU=\bigcup_{x\in\CU}\{\le x\}.
$$
 A subset $\CA$ of $\CV$ is called {\em closed} if the complement $\CV\setminus \CA$ is open.
\end{definition}
Clearly, a subset $\CA$ of $\CV$ is closed if and only if $\CA=\bigcup_{x\in\CA}\{\ge x\}$.  It is also clear that arbitrary unions and arbitrary intersections of closed subsets are closed. The same holds for open subsets.

\subsection{Cofiltered objects in an abelian category $\CA$}

Let $\CA$ be an abelian category. For notational simplicity we assume that the objects in $\CA$ have elements.  
\begin{definition}
\begin{itemize}
\item 
A $(\CV,\le)$-cofiltered object $M$ in $\CA$ is the data of objects $M^{\le x}$  in $\CA$ for all $x\in\CV$ together with morphisms  $M^{\le x}\to M^{\le y}$, $m\mapsto m|_{\le y}$ whenever $y\le x$, such that for all $m\in M^{\le x}$ we have  $m|_{\le z}= (m|_{\le y})|_{\le z}$ if $z\le y\le x$.
\item A homomorphism $f\colon M\to N$ between $(\CV,\le)$-cofiltered objects is the data of morphisms $f^{\le x}\colon M^{\le x}\to N^{\le x}$ such that for all $y\le x$ the diagram 

\centerline{
\xymatrix{
M^{\le x}\ar[d]_{f^{\le x}}\ar[rr]^{m\mapsto m|_{\le y}}&& M^{\le y}\ar[d]^{f^{\le y}}\\ 
N^{\le x}\ar[rr]^{n\mapsto n|_{\le y}}&& N^{\le y}}
}
\noindent
commutes. 
\end{itemize}
\end{definition} 

Let $M$ be a $(\CV,\le)$-cofiltered object and assume that $\CA$ admits arbitrary products. 
For an open subset $\CJ$ of $\CV$ we then define
$$
M^{\CJ}:=\left.\left\{(m_x)\in\prod_{x\in\CJ} M^{\le x}\right|m_x|_{\le y}=m_y\text{ for all $x,y\in \CJ$ with $y\le x$}\right\}.
$$
This is again an object in $\CA$. 
A morphism $f\colon M\to N$ between cofiltered objects induces a morphism $f^{\CJ}\colon M^{\CJ}\to N^{\CJ}$ in $\CA$.

\begin{remarks}
\begin{enumerate}
\item For $\CJ^\prime\subset\CJ$ the projection $\prod_{x\in\CJ} M^{\le x}\to\prod_{x\in\CJ^\prime} M^{\le x}$ with kernel $\prod_{x\in\CJ\setminus\CJ^\prime} M^{\le x}$ induces a homomorphism $M^\CJ\to M^{\CJ^\prime}$ that we denote by $m\mapsto m|_{\CJ^\prime}$. 
\item We have a canonical identification $M^{\{z\mid z\le x\}}=M^{\le x}$, hence there is no danger of contradicting notations. With this identification, $m=(m_x)\in M^{\CJ}$ satisfies $m_x=m|_{\le x}$. 
\item We write $M^{<x}$ instead of $M^{\{z\mid z<x\}}$.
 \end{enumerate}
\end{remarks}

\begin{lemma} \label{lemma-admcofilt}  Let $M$ be a $(\CV,\le)$-cofiltered object in $\CA$. Let $\{\CJ_i\}_{i\in I}$ be a family of open subsets of $\CV$ and set $\CJ=\bigcup_{i\in I}\CJ_i$. Then the homomorphism
\begin{align*}
M^{\CJ}&\to \prod_{i\in I} M^{\CJ_i},\\
m&\mapsto (m|_{\CJ_i})
\end{align*}
is injective with image 
$$
R=\left.\left\{(m_i)\in  \prod_{i\in I} M^{\CJ_i}\right| m_i|_{\CJ_i\cap\CJ_j}=m_j|_{\CJ_i\cap\CJ_j}\text{ for all $i,j\in I$}\right\}.
$$
\end{lemma}
\begin{proof}  As a first step we show that the claimed injectivity holds. So suppose that $m=(m_x)\in M^{\CJ}$ is mapped to $0$. Let $y$ be an element in $\CJ$. Then there is some $i\in\CI$ with $y\in\CJ_i$. Then $m|_{\CJ_i}=0$ implies $m_y=m|_{\le y}=0$. As this holds for all $y\in\CJ$ we deduce  $m=0$. 

 It is clear that the image is contained in  $R$. So let $(m_i)$ be in $R$. For $x\in\CJ$ choose $i\in I$ with $x\in\CJ_i$ and set $m_x:= m_i|_{\le x}$. By definition of $R$, $x\in\CJ_i\cap\CJ_j$ implies $m_i|_{\le x}=m_j|_{\le x}$, so $m_x$ is independent of the choice of $i$.  If $z\le x$, then $z\in\CJ_i$ and $m_z=m_i|_{\le z}=(m_i|_{\le x})|_{\le z}=m_x|_{\le z}$. Hence there is some $m\in M^{\CJ}$ with $m|_{\le x}=m_x$. Now $m|_{\CJ_i}$ and $m_i$ are elements in $M^{\CJ_i}$ with the property $(m|_{\CJ_i})|_{\le x}=m_i|_{\le x}$ for all $x\in\CJ_i$. The injectivity statement, applied to $\CJ_i$ instead of $\CJ$, yields $m|_{\CJ_i}=m_i$. 
\end{proof}

\subsection{Subquotients}
Let $M$ be a $(\CV,\le)$-cofiltered object in $\CA$. For $x\in\CV$ we denote by $M_{[x]}\subset M^{\le x}$ the kernel of the restriction homomorphism $M^{\le x}\to M^{<x}$. This is the subobject of $M^{\le x}$ that contains all  $m$ with $m|_{\le y}=0$ for all $y<x$. 

\begin{definition} We define the {\em support of $M$} as 
$$
\supp_{\le}\, M:=\{x\in\CV\mid M_{[x]}\ne 0\}.
$$
\end{definition}

The following is an easy consequence of the definitions.
\begin{lemma} Let $\CJ\subset\CV$ be an open set with maximal element $x$. Then the restriction homomorphism $M^\CJ\to M^{\le x}$ induces an isomorphism 
$$
\ker(M^{\CJ}\to M^{\CJ\setminus\{x\}})\cong M_{[x]}.
$$
\end{lemma}

\subsection{Flabby cofiltered objects}
Note that we probably should have called a cofiltered object rather a  {\em direct system of objects} in $\CA$. Only the properties stated in the following lemma justify the name {\em cofiltration}. 
\begin{lemma} \label{lemma-surj} The following are equivalent for a $(\CV,\le)$-cofiltered object $M$.
\begin{enumerate}
\item The restrictions $M^{\le w}\to M^{<w}$ are surjective for all $w\in\CV$. 
\item The restrictions $M^{\CJ}\to M^{\CJ\setminus\{w\}}$ are surjective for all open subsets $\CJ$ of $\CV$ with maximal element $w\in\CJ$.
\item The restrictions $M^{\CJ}\to M^{\CJ^\prime}$ are surjective for all  open subsets $\CJ^\prime\subset \CJ$ of $\CV$.
\end{enumerate}
\end{lemma}
\begin{proof} Clearly, (3) implies (2) and (2) implies (1). If we are given a pair $\CJ^\prime\subset\CJ$ of open subsets of $\CV$ we obtain $\CJ^\prime$ from $\CJ$ by removing, step by step, a maximal element. Hence (2) and (3) are equivalent. So we only have to show that (1) implies (2). 

So let $\CJ$ be an open subset of $\CV$ and let $\CJ^\prime=\CJ\setminus\{w\}$ for a maximal element $w\in\CJ$. 
Let $m^\prime \in M^{\CJ^\prime}$. Then there is an element $\tilde m\in M^{\le w}$ that is a preimage of $m^\prime|_{<w}$. We define $m=(m_x)\in\prod_{x\in\CJ} M^{\le x}$ by $m_x=m^\prime|_{\le x}$ for $x\ne w$ and $m_w=\tilde m$. It is easy to see that for $x,y\in\CJ$ and $y\le x$ we have $m_x|_{\le y}=m_y$. Hence $m$ defines an element in $M^{\CJ}$ which is, by construction, a preimage of $m^\prime$.
\end{proof}

\begin{definition} We say that the $(\CV,\le)$-cofiltered object $M$ is {\em flabby} if the equivalent conditions in Lemma \ref{lemma-surj} hold.
\end{definition}

\subsection{Cofiltered $\lCZ$-modules}
We now assume that $\CA$ is the category of graded $\lCZ$-modules and we consider $(\CV,\le)$-cofiltered $\lCZ$-modules.  Let $M$ be a $(\CV,\le)$-cofiltered $\lCZ$-module.  Here is an additional assumption on $M$. We call it the {\em support condition}. Note that we have now two notions of support for $M$: one comes from the cofiltration and is denoted by $\supp_\le\, M$, the other comes from the $\lCZ$-module structure and is denoted by $\supp_{\lCZ} M$.

\begin{itemize}
\item[(S)]  For each $\Theta\in\lCV$ and $x\in\Theta$, the support of the $\lCZ$-module $M_{[x]}$ (cf. Definition \ref{def-suppZmod}) is contained in  $\Theta$.  
\end{itemize}

Recall that we define for each $\lCZ$-module $N$ the generic decomposition $N\subset \bigoplus _{\Omega\in\lCV}N^{\Omega}$ in Section \ref{subsec-gendec}. Then condition (S) means that we have  
$$
(M_{[x]})^\Omega=\begin{cases}
M_{[x]},&\text{ if $x\in\Omega$}\\
0,&\text{ if $x\not\in\Omega$}.
\end{cases}
$$
In particular, the canonical homomorphism $M_{[x]}\to M^{\le x}\to (M^{\le x})^\Omega$ is injective, and
$$
(M^{\le x})^\Theta=(M^{<x})^{\Theta}
$$
if $x\not\in\Theta$.

\begin{definition}
\begin{itemize}
\item  We denote by $\lCZ\catmod^{\le}$ the category of $(\CV,\le)$-cofiltered graded $\lCZ$-modules $M$ that satisfy the following assumptions:
\begin{itemize}
\item For every $x\in\CV$ the object $M^{\le x}$ is torsion free over $S$ and finitely generated over $S$. 
\item $M$ is is flabby and satisfies condition (S).
\end{itemize}  
\item   
We denote by $\CC^{\refl}$ the full subcategory of $\lCZ\catmod^{\le}$ that contains  all objects that are reflexive, i.e. that have the property that for each open subset $\CJ$ of $\CV$ the $S$-module $M^\CJ$ is reflexive.
\item   
We denote by $\CC$ the full subcategory of $\lCZ\catmod^{\le}$ that contains  all objects that {\em admit a Verma flag}, i.e. that have the property that for each open subset $\CJ$ of $\CV$ the $S$-module $M^\CJ$ is graded free.

\end{itemize}
 \end{definition} 

For any $w\in\CV$ we define the {\em standard object} $\Delta(w)$ in the category $\CC^{\refl}$ as follows: Suppose that $w\in\Theta$. Let $\lDelta(w)$ be the $\lCZ$-module that is free of graded rank $1$ as an $S$-module and on which $(z_\Omega)$ acts as multiplication with $z_\Theta$. Then we set

$$
\Delta(w)^{\le x}=\begin{cases}
\lDelta(w),&\text{ if $w\le x$},\\
0,&\text{ else}
\end{cases}
$$
with restriction maps 
$$
\left(\Delta(w)^{\le x}\to\Delta(w)^{\le y}\right)=\begin{cases}
\id_{\lDelta(w)},&\text{ if $w\le x$, $w\le y$},\\
0,&\text{ else.}
\end{cases}
$$
It is clear that $\Delta(w)$ is flabby and satisfies condition (S).

\subsection{The Fitting decomposition and the Krull-Schmidt theorem}

Let $M$ be an object in $\lCZ\catmod^{\le}$ and $f\colon M\to M$  a morphism in $\lCZ\catmod^{\le}$. Then $M^{\le x}$ is a finitely generated $\DZ$-graded $S$-module, so each homogeneous component is a finite dimensional $k$-vector space. We obtain a Fitting decomposition
$$
M^{\le x}=\ker^\infty f^{\le x}\oplus\im^{\infty} f^{\le x},
$$
where  
\begin{align*}
 \ker^\infty f^{\le x}&:=\{m\in M^{\le x}\mid (f^{\le x})^n(m)=0\text{ for some $n\gg 0$}\}=\bigcup_{n\in\DN}\ker (f^{\le x})^n,\\
 \im^\infty f^{\le x}&:=\bigcap_{n\in\DN}\im (f^{\le x})^n.
\end{align*}
Clearly, $\ker^\infty f^{\le x}$ and $\im^\infty f^{\le x}$ are sub-$\lCZ$-modules of $M^{\le x}$. The functoriality of this decomposition assures that the above data defines a direct sum decomposition $M=\ker^\infty f\oplus\im^\infty f$ of $(\CV,\le)$-cofiltered $\lCZ$-modules and even of objects in $\lCZ\catmod^{\le}$. From this  decomposition and the fact that $M$ is finitely generated as an $S$-module one deduces the Krull-Schmidt Theorem for $\lCZ\catmod^{\le}$ with standard arguments:

\begin{theorem}\label{thm-KrullSchm} 
\begin{enumerate}
\item Each $M\in\lCZ\catmod^{\le}$ admits an up to isomorphism and rearrangement unique decomposition into indecomposables.
\item Each endomorphism of an indecomposable object $M$ in $\lCZ\catmod^{\le}$ is either an automorphism or nilpotent. In particular, $\End_{\lCZ\catmod^{\le}}(M)$ is a local ring. 
\end{enumerate}
\end{theorem}

\subsection{The extension spaces $M^{\delta x}$}
Let $M$ be an object in $\lCZ\catmod^{\le}$.
For any $x\in\CV$ we define a $\lCZ$-module $M^{\delta x}$ as follows. Suppose that $x\in\Omega\in\lCV$. Let $M^{\le x}\to (M^{\le x})^\Omega$ be the projection onto the $\Omega$-component. Then, by the support condition, the composition
$$
M_{[x]}\to M^{\le x}\to (M^{\le x})^\Omega
$$
is injective. We define
$$
M^{\delta x}:= (M^{\le x})^\Omega/M_{[x]}.
$$
We then have natural homomorphisms
\begin{align*}
d_x\colon (M^{\le x})^\Omega&\to M^{\delta x},\\
u_x\colon M^{<x}=M^{\le x}/M_{[x]}&\to M^{\delta x}.
\end{align*}

\begin{lemma}\label{lemma-extspac} The homomorphism $M^{\le x}\to (M^{\le x})^\Omega\oplus M^{<x}$,  $m\mapsto (m_\Omega, m|_{<x})$,  is injective with image
$R=\{(m_\Omega,m^\prime)\mid d_x(m_\Omega)=u_x(m^\prime)\}$.
\end{lemma}
\begin{proof} We first show that the homomorphism is indeed injective. Suppose that $m\in M^{\le x}$ is in its kernel. Then $m|_{<x}$ is zero, hence $m\in M_{[x]}$. Now the claim follows from the injectivity of the composition $M_{[x]}\to M^{\le x}\to (M^{\le x})^\Omega$. 

That the image of the homomorphism in the lemma is contained in $R$ follows from the fact that the diagram

\centerline{
\xymatrix{
&(M^{\le x})^\Omega\ar[dr]&\\ 
M^{\le x}\ar[ur]\ar[dr]&& M^{\delta x}\\
&M^{<x}\ar[ur]&
}
}
\noindent
commutes. In order to finish the proof, we have to show that for each $(m_\Omega,m^\prime)$ with $d_x(m_\Omega)=u_x(m^\prime)$ there is an $m\in M^{\le x}$ with $\Omega$-component $m_\Omega$ and whose restriction to $\{<x\}$ is $m^\prime$. As $M^{\le x}\to M^{<x}$ is surjective, we can assume, without loss of generality, that $m^\prime=0$. But then $d_x(m_\Omega)=0$, which implies that $m_\Omega\in M_{[x]}\subset M^{\le x}$, which finishes the proof. 
\end{proof}

\subsection{An extension result}
The following is a rather technical result that we use later on.
\begin{lemma}\label{lemma-locext} Let $M\in\CC^{\refl}$, let $\Theta$ be a vertex of $\lCG$ and let $x\in\Theta$. Suppose that  $m^\prime\in M^{<x}$ satisfies  $m^\prime_\Theta=0$ and $m^\prime_\Omega\in\alpha (M^{<x})^{\Omega}$ if $\Omega$ is connected to $\Theta$ by an edge with label $\alpha$.  Then $u_x(m^\prime)=0$.
\end{lemma}
\begin{proof} By Lemma \ref{lemma-extspac} the statement $u_x(m^\prime)=0$ is equivalent to the statement that there is an element $m\in M^{\le x}$ with $m_\Theta=0$ and $m|_{<x}=m^\prime$. Let us try define $m=(m_\Omega)\in \bigoplus_{\Omega} (M^{\le x})^{\Omega}$ by
\begin{itemize}
\item $m_\Theta =0$,
\item $m_\Omega=m^\prime_\Omega$ for all $\Omega\ne\Theta$.  
\end{itemize}
(Note that $(M^{\le x})^{\Omega}=(M^{< x})^{\Omega}$ for $\Omega\ne\Theta$).
So we need to show that $m$ is contained in $M^{\le x}\subset \bigoplus_{\Omega} (M^{\le x})^\Omega$. 

As $M^{\le x}$ is reflexive it suffices to prove that $m\in (M^{\le x})^L$ for each saturated sublattice $L$ of rank $1$. Now if $\alpha\not\in L$ for all edges $\Gamma\stackrel{\alpha}\linie \Omega$, then $((M^{\le x})^{\Omega})^L$ is a direct summand in $(M^{\le x})^L$ and the statement follows. Otherwise we can assume $L=\DZ\alpha$. By the GKM-condition there is a unique edge $\Gamma\stackrel{\alpha}\linie \Omega$ labeled by $\alpha$, and there is a unique direct summand $((M^{\le x})^L)^\prime$ with support contained in $\{\Omega,\Gamma\}$. Moreover, the local structure algebra $\{(z_\Omega,z_\Gamma)\mid z_\Omega\equiv z_\Gamma\mod\alpha\}$ acts on $((M^{\le x})^L)^\prime$. By our assumption there is some $\tilde m=(\tilde m_\Omega,\tilde m_\Gamma)\in ((M^{\le x})^ L)^\prime$ with $m_\Gamma=\alpha \tilde m_\Gamma$. Then $(\alpha,0) \tilde m=m$, which proves our claim.
\end{proof}

 \subsection{An exact structure on $\CC^{\refl}$}
Let $0\to M\stackrel{f}\to N\stackrel{g}\to O\to 0$ be a sequence in $\CC^{\refl}$. 
\begin{definition} We say that the above sequence is {\em exact}, if for any  open subset $\CJ$ the induced sequence
$$
0\to M^\CJ\stackrel{f^\CJ}\longrightarrow N^\CJ\stackrel{g^\CJ}\longrightarrow O^\CJ\to 0
$$
is an exact sequence of $S$-modules. In this case we call $g$ an {\em admissible epimorphism} and $f$ an {\em admissible monomorphism}.
 \end{definition}
 
Arguments similar to those in the proof of \cite[Theorem 4.1]{FieAdv} show that this indeed gives an exact structure in the sense of Quillen.

\subsection{Projective objects in $\CC^{\refl}$}

The well-known notion of a projective object in an abelian category carries over to categories with an exact structure.
\begin{definition} An object $P$ of $\CC^{\refl}$ is called {\em projective}, if the functor 
$$
\Hom_{\CC^{\refl}}(P,\cdot)\colon \CC^{\refl}\to\DZ\catmod
$$ 
maps short exact sequences to short exact sequences. 
 \end{definition}
 
We will use the following result later in order to show that the objects constructed with the Braden--MacPherson algorithm are projective in $\CC^{\refl}$.
 
\begin{proposition} \label{prop-projobj} Let $B$ be an object in $\CC^{\refl}$ with the following properties:
\begin{enumerate}
\item For each $\Omega\in \lCV$ and each $x\in\Omega$, the object $(B^{\le x})^{ \Omega}$ is a projective graded $S$-module.
\item For each $M\in\CC^{\refl}$ and all $x\in\CV$,  the homomorphism
$\Hom(B,M^{\le x})\to  \Hom(B,M^{<x})$ induced by the restriction homomorphism $M^{\le x}\to M^{<x}$, is surjective.
\end{enumerate}
Then $B$ is projective.
\end{proposition}
\begin{proof} Let $\pi\colon M\sur N$ be an admissible epimorphism in $\CC^{\refl}$ and let $f\colon B\to N$ be a morphism. We  need to construct homomorphisms $g^{\le x}\colon B^{\le x}\to M^{\le x}$ that are compatible with the restriction homomorphisms $(\cdot)^{\le x}\to(\cdot)^{\le y}$ and, moreover, satisfy $\pi^{\le x}\circ g^{\le x}=f^{\le x}$. We construct these homomorphisms by induction on $x$.  If $x$ is not contained in the support of $M$ we set $g^{\le x}=0$. So suppose that $g^{\le y}\colon B^{\le y}\to M^{\le y}$ is already constructed for all $y<x$. We then have an induced homomorphism $g^{<x}\colon B^{<x}\to M^{<x}$. By assumption (2) we can find a homomorphism $g^\prime\colon B^{\le x}\to M^{\le x}$ with $(g^\prime)^{<x}=g^{<x}$. Hence $f^{\le x}-\pi\circ g^\prime$ is such that $(f^{\le x}-g^\prime)^{<x}=0$, i.e. its image is contained in $N_{[x]}$ and it factors over the projection $B^{\le x}\to (B^{\le x})^{\Omega}$ onto the stalk. By the projectivity of this stalk in the category of graded $S$-modules (assumption (1)) and since $\pi_{[x]}\colon M_{[x]}\to N_{[x]}$ is surjective, we can find a homomorphism $h\colon B^{\le x}\to M_{[x]}$ such that 

\centerline{
\xymatrix{
M^{\le x}\ar[r]^{\pi^{\le x}}&N^{\le x}\\
M_{[x]}\ar[u]\ar[r]^{\pi_{[x]}}&N_{[x]}\ar[u]\\ 
&B^{\le x}\ar[u]_{f^{\le x}-\pi\circ g^\prime}\ar[ul]^h
}
}\noindent
commutes. So $g^\prime+h\colon B^{\le x}\to M^{\le x}$ is what we are looking for.
\end{proof}

\section{Filtered sheaves}
We keep the data $\CG,\CV,\le, \lCG,\lCV,k,\lCZ,\dots$ of the previous sections. Now we  introduce the notion of $(\CV,\le)$-cofiltered sheaves on $\lCG$.

\begin{definition} A {\em $(\CV,\le)$-cofiltered sheaf} $\SM$ on $\lCG$ is given by the following data:
\begin{itemize}
\item A $(\CV,\le)$-cofiltered graded $S$-module $\SM^\Omega$ for each vertex $\Omega$ of $\lCG$ such that $\supp_{\le}\,\SM^{\Omega}\subset\Omega$.
\item A  $(\CV,\le)$-cofiltered graded $S$-module $\SM^E$ for each edge $E\colon \Omega\linie \Omega^\prime$ of $\lCG$ with $\alpha(E)\SM^E=\{0\}$ and such that $\supp_{\le}\,\SM^E\subset \supp_{\le}\, \SM^{\Omega}\cup\supp_{\le}\, \SM^{\Omega^\prime}$.
\item A homomorphism $\rho_{\Omega,E}=\rho^\SM_{\Omega,E}\colon \SM^\Omega\to\SM^E$ of cofiltered graded $S$-modules whenever the vertex $\Omega$ lies on the edge $E$.
\end{itemize}
A morphism $f\colon \SM\to\SN$ between the sheaves $\SM$ and $\SN$ on $\CG$ is given by homomorphisms $f^\Omega\colon \SM^\Omega\to\SN^\Omega$ and $f^E\colon\SM^E\to\SN^E$ of cofiltered graded  $S$-modules for all vertices $\Omega$ and all edges $E$ such that the diagram 

\centerline{
\xymatrix{
 \SM^\Omega\ar[d]_{\rho^\SM_{\Omega,E}}\ar[r]^{f^\Omega}&\SN^\Omega\ar[d]^{\rho^\SN_{\Omega,E}}\\
\SM^E\ar[r]^{f^E}&\SN^E 
}
}
\noindent
commutes whenever the vertex $\Omega$ is adjacent to the edge $E$.  
\end{definition}


\begin{remark} Note that we do not assume that the filtrations on $\SM^{\Omega}$ and $\SM^{E}$ are flabby.
\end{remark}

\begin{definition} The support of a cofiltered sheaf $\SM$ is the union of the supports of its stalks, i.e.
$$
\supp_{\le}\,\SM=\bigcup_{\Omega}\supp_{\le}\,\SM^\Omega.
$$
\end{definition}

 \subsection{Sections of filtered sheaves}

Let  $\SM$ be a cofiltered sheaf on $\lCG$ and $\CJ$ an open subset of $\CV$.  
 
 \begin{definition} The {\em space of local sections} of $\SM$ over $\CJ$ is
$$
\Gamma(\CJ,\SM):=\left.\left\{(m_\Omega)\in\bigoplus_{\Omega\in\lCV}(\SM^\Omega)^\CJ\right| \begin{matrix}
\text{ $\rho^\CJ_{\Omega,E}(m_\Omega)=\rho^\CJ_{\Omega^\prime,E}(m_{\Omega^\prime})$}\\
\text{ for all edges $E\colon \Omega\llinie\Omega^\prime$}
\end{matrix}\right\}.
$$
\end{definition}

For open subsets $\CJ^\prime\subset\CJ$ the homomorphisms $(\SM^{\Omega})^\CJ\to (\SM^{\Omega})^{\CJ^\prime}$ on the stalks induce a homomorphism $\Gamma(\CJ,\SM)\to\Gamma(\CJ^\prime,\SM)$ that we denote by $m\mapsto m|_{\CJ^\prime}$.

\begin{lemma} \label{lemma-glue} Let $\SM$ be a $(\CV,\le)$-cofiltered sheaf, let $\{\CJ_i\}_{i\in I}$ be a family of open subsets of $\CV$ and let $\CJ=\bigcup_{i\in I}\CJ_i$. 
Then the homomorphism 
\begin{align*}
\Gamma(\CJ,\SM)&\to\prod_{i\in I}\Gamma(\CJ_i,\SM),\\
m&\mapsto (m|_{\CJ_i})
\end{align*}
is injective with image
$$
R=\left\{(m_i)\in\prod_{i\in I}\Gamma(\CJ_i,\SM)\left|\begin{matrix}\text{ $m_i|_{\CJ_i\cap\CJ_j}=m_j|_{\CJ_i\cap\CJ_j}$}\\ \text{ for all $i,j\in i$}\end{matrix}\right\}\right..
$$
\end{lemma}
\begin{proof}  
Suppose that $m=(m_\Omega)$ is in the kernel, i.e. $m|_{\CJ_i}=0$ for all $i\in I$. Suppose that $m_\Omega\ne 0$. Then there is some $x\in\Omega$ with $(m_{\Omega})|{\le x}\ne 0$. Choose $i\in I$ such that $x\in\CJ_i$. Then $(m_\Omega)|_{\CJ_i}\ne 0$, which contradicts our assumption $m|_{\CJ_i}=0$. 

It is, moreover, clear, that the image of the above map is contained in $R$. Conversely, suppose that $m_i\in R$. For a vertex $\Omega$ we obtain an $I$-tuple $((m_i)_\Omega)\in\prod_{i\in I}\SM^{\Omega,\CJ_i}$ with the property $(m_i)_\Omega|_{\CJ_i\cap\CJ_j}=(m_j)_\Omega|_{\CJ_i\cap\CJ_j}$ for all $i,j\in I$. By Lemma \ref{lemma-admcofilt}, this  uniquely defines an element $m_\Omega\in \SM^{\Omega,\CJ}$. Now we only need to show that $\rho^\CJ_{\Omega,E}(m_\Omega)=\rho^\CJ_{\Omega^\prime,E}(m_{\Omega^\prime})$ for all edges $E\colon\Omega\llinie\Omega^\prime$. But suppose that this fails for some edge $E$. Then there must be some $i\in I$ such that $\rho^{\CJ_i}_{\Omega,E}(m_\Omega|_{\CJ_i})\ne \rho^{\CJ_i}_{\Omega^\prime,E}(m_{\Omega^\prime}|_{\CJ_i})$. But this means that $\rho^{\CJ_i}_{\Omega,E}((m_i)_\Omega)\ne \rho^{\CJ_i}_{\Omega^\prime,E}((m_i)_{\Omega^\prime})$, which contradicts the fact that $m_i\in\Gamma(\CJ_i,\SM)$.
\end{proof}

\subsection{Extension spaces for sheaves}

Let $\SM$ be a filtered sheaf on $\lCG$. Let $\Theta$ be a vertex of $\lCG$ and $z\in\Theta$.  Let $\CE_\Theta$ denote the set of edges of $\lCG$ that are adjacent to $\Theta$.  We define 
\begin{align*}
d_z\colon\SM^{\Theta,\le z}&\to\left(\bigoplus_{E\in\CE_\Theta}\SM^{E,\le z}\right)\oplus\SM^{\Theta,<z},\\
m&\mapsto\left((\rho^{\le z}_{\Theta,E}(m))_{E\in\CE_{\Theta}},m|_{<z}\right).
\end{align*}
We also define
\begin{align*}
u_z\colon \Gamma(<z,\SM)&\to\left(\bigoplus_{E\in\CE_\Theta}\SM^{E,\le z}\right)\oplus \SM^{\Theta,<z},\\
(m_\Omega)&\mapsto \left((\rho_{\Omega,E}^{< z}(m_\Omega))_{E\in\CE_\Theta}, m_\Theta\right).
\end{align*}
Note that above we used the fact that for all $\Omega$ with $z\not\in\Omega$ we have $\SM^{\Omega,<z}=\SM^{\Omega,\le z}$, as $z\not\in\supp_{\le}\,\SM^{\Omega}$. We denote the image of $u_z$ by $\SM^{\delta z}$.

It might be instructive for the reader to compare the next statement with  Lemma \ref{lemma-extspac}.
\begin{lemma} For each $z\in\CV$ the homomorphism $\Gamma(\le z,\SM)\to\SM^{\Theta,\le z}\oplus\Gamma(<z,\SM)$, $m\mapsto (m_\Theta, m|_{<z})$ is injective with image 
$$
R=\{(m_1,m_2)\in \SM^{\Theta,\le z}\oplus \Gamma(<z,\SM)\mid d_z(m_1)=u_z(m_2)\}.
$$
\end{lemma}

\begin{proof} It is clear that the homomorphism is injective and that its image is contained in the prescribed set.  Conversely, let $(m_1,m_2)\in \SM^{\Omega,\le z}\oplus \Gamma(<z,\SM)$ be such that $d_z(m_1)=u_z(m_2)$.  Define $m=(m_\Omega)\in\bigoplus\SM^{\Omega,\le z}$ by $m_\Theta=m_1$ and $m_\Omega=(m_2)_\Omega$ if $\Omega\ne\Theta$ (note that then $\SM^{\Omega,\le z}=\SM^{\Omega,<z}$). 

We claim that $m\in\Gamma(\le z,\SM)$, which yields the statement of the lemma, as $m_\Omega|_{<z}=(m_2)_{\Omega}$ for $\Omega\ne\Theta$ and  $m_1|_{<z}=(m_2)_\Theta$ as $d_z(m_1)=u_z(m_2)$. 

So let $E\colon \Omega\llinie\Omega^\prime$ be an edge. If $\Theta\ne\Omega$ and $\Theta\ne\Omega^\prime$, then 
$$
\rho_{\Omega,E}^{\le z}(m_\Omega)=\rho_{\Omega,E}^{< z}((m_2)_\Omega)=\rho_{\Omega^\prime,E}^{<z}((m_2)_{\Omega^\prime})=\rho_{\Omega^\prime,E}^{\le z}(m_{\Omega^\prime}).
$$
Otherwise, we can assume without loss of generality that $\Omega^\prime=\Theta$. But in this case 
$$
\rho^{\le z}_{\Omega,E}((m_2)_\Omega)=\rho^{\le  z}_{\Omega,E}(m_1)
$$
as $d_z(m_1)=u_z(m_2)$.
 \end{proof}

\subsection{Flabby sheaves}

Here is the sheaf theoretic analogue of Lemma \ref{lemma-surj}.

\begin{lemma}\label{lemma-flab} Let $\SM$ be a cofiltered sheaf. Then the following are equivalent:
\begin{enumerate}
\item For each pair $\CJ^\prime\subset\CJ$ of open subsets of $\CV$, the homomorphism $\Gamma(\CJ,\SM)\to\Gamma(\CJ^\prime,\SM)$ is surjective.
\item For each open subset $\CJ$ of $\CV$ with maximal element $x$ the homomorphism $\Gamma(\CJ,\SM)\to\Gamma(\CJ\setminus\{x\},\SM)$ is surjective.

\item For each $z\in\lCV$, the homomorphism $\Gamma(\le z,\SM)\to\Gamma(<z,\SM)$ is surjective.
\item For each $z\in\CV$, the  image of $u_z$ is contained in the image of $d_z$.
\end{enumerate}
\end{lemma}
\begin{proof} It is clear that (1) and (2) are equivalent. So let us show that (2) and (3) are equivalent. Clearly, (2) implies (3). So assume (3) and suppose that $\CJ$ is an  open subset of $\CV$ with maximal element $x\in\CJ$. Let $m^\prime$ be a section in $\Gamma(\CJ\setminus\{x\},\SM)$. By assumption there is a section $\tilde m\in\Gamma(\le x,\SM)$ with the property that $\tilde m|_{<x} =m^\prime|_{<x}$. By Lemma \ref{lemma-glue} there is an element $m\in\Gamma(\CJ,\SM)$ with $m|_{\le x}=\tilde m$ and $m|_{\CJ\setminus\{x\}}=m^\prime$. In particular, this is an extension of $m^\prime$.

We show now that (3) and (4) are equivalent.  Let $z\in\CV$ and suppose that $z\in\Theta$. Let $m\in\Gamma(\le z,\SM)$ and $m^\prime\in\Gamma(<z,\SM)$. Then $m$ is a preimage of $m^\prime$ if and only if $d_z(m_\Theta)=u_z(m^\prime)$ and $m_\Omega=m^\prime_\Omega$ for all $\Omega\ne\Theta$. 
\end{proof}

\begin{definition} We say that a cofiltered sheaf $\SM$ on $\lCG$ is {\em flabby} if the equivalent conditions of Lemma \ref{lemma-flab} are satisfied.
\end{definition}

\subsection{A condition for reflexivity}

Let $\SM$ be a cofiltered sheaf on $\lCG$.
\begin{lemma} \label{lemma-refl} Suppose that for any $x\in\CV$, the $S$-module $\SM^{\Omega,\le x}$ is reflexive, and that for any edge $E$ labeled by $\alpha$ the $S/\alpha S$-module $\SM^{E,\le x}$ is reflexive. Then the $S$-module  $\Gamma(\CJ,\SM)$ of local sections is reflexive for any open subset $\CJ$ of $\CV$. 
\end{lemma}
\begin{proof} Let $\CJ$ be an open subset of $\CV$. Then  $\prod_{x\in\CV}\SM^{\Omega,\le x}$ is a reflexive $S$-module, and it follows from the definition of the submodule $\SM^{\Omega,\CJ}$ that it  is reflexive as well. Analogously one argues that $\SM^{E,\CJ}$ is a reflexive $S/\alpha S$-module. Then the statement follows from Lemma 4.7 in \cite{FieAdv}.
\end{proof}
\section{A filtered Braden--MacPherson algorithm}

We will now give a cofiltered version of the {\em canonical sheaf} of Braden and MacPherson (cf. \cite{BMP}). Its defining properties are listed in the following Theorem.

\begin{theorem} \label{thm-filtBMP} For each $w\in\CV$ there is an up to isomorphism unique $(\CV,\le)$-cofiltered sheaf $\SB(w)$ on $\lCG$ with the following properties:
\begin{enumerate}
\item For each $\Omega\in\lCV$  and each $x\in\Omega$ the image of $d_x$ is $\SB(w)^{\delta x}$ and the resulting  homomorphism $\SB(w)^{\Omega,\le x}\to\SB(w)^{\delta x}$ is a projective cover in the category of graded $S$-modules.
\item  The support of $\SB(w)$ is contained in $\{\ge w\}$ and 
$$
\SB(w)^{\Omega,\le w}=\begin{cases} 
S,&\text{ if $w\in\Omega$},\\ 
0,&\text{ if $w\not\in\Omega$}.
\end{cases}
$$
\item Let $E\colon\Theta\stackrel{\alpha}\llinie\Omega$ be an edge. If $z\in\Omega$, then $\rho_{\Theta,E}^{\le z}\colon\SB(w)^{\Theta,\le z}\to\SB(w)^{E,\le z}$ is surjective with kernel $\alpha\SB(w)^{\Theta,\le z}$. 
\end{enumerate}
\end{theorem}
\begin{proof} We first prove the uniqueness statement. In fact, we prove the following stronger statement:
\begin{lemma} Let $w\in\CV$ and let $\Omega\in\lCV$ be such that $w\in\Omega$. Suppose that $\SC$ and $\SD$ are cofiltered sheaves on $\lCG$ that satisfy the properties (1), (2) and (3) listed in Theorem \ref{thm-filtBMP}, and suppose that we have an isomorphism $h\colon \SC^{\Omega,\le w}\to \SD^{\Omega,\le w}$. Then there is an isomorphism   $f\colon \SC\to\SD$ of cofiltered sheaves such that $f^{\Omega,\le w}=h$. 
\end{lemma} 
\begin{proof}[Proof of the Lemma] We construct the components $f^{\Omega,\le z}$ and $f^{E,\le z}$ inductively.
 For all $z\in\CV$ with $z\not\ge w$ and all $\Omega$ we have $\SC^{\Omega,\le z}=\SD^{\Omega,\le z}=0$, hence  $f^{\Omega,\le z}:=0$. For all edges $E$ we have, by property (3), $\SC^{E,\le z}=\SD^{E,\le z}=0$, hence  $f^{E,\le z}:=0$.  For a vertex $\Theta\ne \Omega$ we have $\SC^{\Theta,\le w}=0=\SD^{\Theta,\le w}$, hence $f^{\Theta,\le w}:=0$. Then we define $f^{\Omega,\le w}:=h$. It is obvious that so far we had no choice in the definitions, that the data we defined is compatible 
 with both the restriction and the $\rho$-maps, and that so far we only defined isomorphisms.

Now let $z>w$ and suppose that we have already constructed isomorphisms  $f^{\Omega,\le x}$ and $f^{E,\le x}$ for all vertices $\Omega$ and all edges $E$ and all $x<z$ and that these data are compatible with the restriction and the $\rho$-maps. We then also have  induced isomorphisms $f^{\Omega,<z}$ and $f^{E,<z}$ for all $\Omega$ and $E$. Let $\Omega$ be a vertex. If $z\not\in\Omega$, then we have $\SC^{\Omega,\le z}=\SC^{\Omega,<z}$ and $\SD^{\Omega,\le z}=\SD^{\Omega,<z}$, hence we have to define  $f^{\Omega,\le z}=f^{\Omega,<z}$.  Let $E\colon\Omega\llinie\Omega^\prime$ be an edge. If $z\not\in\Omega\cup\Omega^\prime$, then $\SC^{E,\le z}=\SC^{E,<z}$ and $\SD^{E,\le z}=\SD^{E,<z}$ and we have to define $f^{E,\le z}=f^{E,<z}$. 

Now suppose that $E\colon\Omega\llinie\Theta$ is an edge and $z\in\Omega$.  Then the isomorphism $f^{\Theta,\le z}\colon\SC^{\Theta,\le z}\to \SD^{\Theta,\le z}$, that we have already constructed, induces, by property (3), an isomorphism $\SC^{E,\le z}\to\SD^{E,\le z}$ that we have to take for $f^{E,\le z}$.  Finally, the isomorphisms that we have constructed so far yield an isomorphism $f\colon\SC^{\delta z}\to \SD^{\delta z}$. As $d_z^\SC\colon\SC^{\Omega,\le z}\to\SC^{\delta z}$ and $d_z^\SD\colon\SD^{\Omega,\le z}\to\SD^{\delta z}$ are projective covers in the category of graded $S$-modules there is an isomorphism  $f^{\Omega,\le z}$ such that the diagram 

\centerline{
\xymatrix{
\SC^{\Omega,\le z}\ar[d]_{d_z^{\SC}} \ar[r]^{f^{\Omega,\le z}} &\SD^{\Omega,\le z}\ar[d]_{d_z^{\SD} }\\
\SC^{\delta z}\ar[r]^{f^{\delta z}} &\SD^{\delta z}
}
}
\noindent
 commutes. This also ensures that the homomorphism $f^{\Omega,\le z}$ is compatible with the homomorphisms $\rho^{\SC,\le z}_{\Omega, E}$ and $\rho^{\SD,\le z}_{\Omega, E}$ for all edges $E$ that are adjacent to $\Omega$, and also with the  homomorphisms $f^{\Omega,\le x}\colon\SC^{\Omega,\le x}\to\SD^{\Omega,\le x}$ for all $x\le z$. 
So we defined indeed an extension, which is, in addition, an isomorphism. 
\end{proof}

Now we  carry on with the proof of the theorem and deal with the existence part. 
We will describe an algorithm that constructs  a cofiltered sheaf $\SB$. From the construction it will be immediately clear that the sheaf $\SB$ satisfies properties (1), (2) and (3).
For convenience, we will stipulate the following. If $\CJ$ is an open subset of $\CV$, then we say that {\em we have constructed everything for $\CJ$} if we have constructed for all vertices $\Omega$ and all edges $E$
\begin{itemize}
\item  the objects $\SB^{\Omega,\le x}$ and $\SB^{E,\le x}$ for all $x\in\CJ$, 
\item the homomorphisms $\SB^{\Omega,\le x}\to \SB^{\Omega,\le y}$ and $\SB^{E,\le x}\to \SB^{E,\le y}$ for all  $x,y\in\CJ$ with $y\le x$,
\item the homomorphisms $\rho^{\le x}_{E,\Omega}\colon\SB^{\Omega,\le x}\to\SB^{E,\le x}$ for all $x\in\CJ$ if $\Omega$ is a vertex of $E$. 
\end{itemize}
and if the data above satisfies all possible compatibility conditions for the sheaf data. 
In what follows, we describe the algorithm.

\noindent{\bf Step 1}
For $x\not\ge w$  and all vertices $\Omega\in\lCV$ and edges $E$ we set
$$
\SB^{\Omega,\le x}:=0,\quad \SB^{E,\le x}:=0.
$$
Then we set 
$$
\SB^{\Omega,\le w}:=
\begin{cases}
S,&\text{if $w\in\Omega$},\\
0,&\text{if $w\not\in\Omega$}
\end{cases}
$$
and
$$
\SB^{E,\le w}:=0
$$
for all edges $E$. The restriction and $\rho$-maps are the zero maps. 

\noindent
{\bf Step 2}
Let $\CJ$ be an open subset of $\CV$ and suppose that we have constructed everything for $\CJ$. Then we can already define
$$
\SB^{\Omega,\CJ}:=\left \{(b_y)\in\prod_{y\in\CJ}\SB^{\Omega,\le y}\mid b_y|_{\le z}=b_z\text{ for all $y\in\CJ$ and $z\le y$}\right\}
$$
and
$$
\SB^{E,\CJ}:=\left \{(b_y)\in\prod_{y\in\CJ}\SB^{E,\le y}\mid b_y|_{\le z}=b_z\text{ for all $y\in\CJ$ and $z\le y$}\right\}.
$$
 For $\CJ^\prime\subset\CJ$ the projection  $\prod_{y\in\CJ}\SB^{\Omega,\le y}\to\prod_{y\in\CJ^\prime}\SB^{\Omega,\le y}$ with kernel $\prod_{y\in\CJ\setminus\CJ^\prime}\SB^{\Omega,\le y}$  induces a homomorphism $\SB^{\Omega,\CJ}\to\SB^{\Omega,\CJ^\prime}$. Analogously  we obtain $\SB^{E,\CJ}\to\SB^{E,\CJ^\prime}$. 
The product $\prod_{y\in\CJ}\rho^{\le y}_{\Omega, E}$ induces a homomorphism
$$
\rho^{\CJ}_{\Omega,E}\colon \SB^{\Omega,\CJ}\to\SB^{E,\CJ}.
$$

\noindent
{\bf Step 3}
Let $z\in\CV$ and suppose that we have constructed everything for $\{< z\}$. We will now construct everything for $\{\le z\}$.  From Step 2 we obtain  $\SB^{\Omega,<z}$ and $\SB^{E,<z}$ and the homomorphism $\rho_{\Omega,E}^{<z}\colon \SB^{\Omega,<z}\to\SB^{E,<z}$.
 Let $E\colon\Omega\stackrel{\alpha}\llinie\Omega^\prime$ be an edge. If $z\not\in\Omega\cup\Omega^\prime$, then we set
$$
\SB^{E,\le z}:=\SB^{E,< z}
$$ 
and we let $\SB^{E,\le z}\to\SB^{E,<z}$ be the identity. This guarantees
$$
\supp_{\le}\, \SB^E\subset \supp_{\le}\, \SB^{\Omega,\le z}\cup\supp_{\le}\,\SB^{\Omega', \le z}.
$$
Suppose $z\in\Omega\cup \Omega^\prime$. As the roles of $\Omega$ and $\Omega^\prime$ are symmetric here, we can assume $z\in\Omega^\prime$. Then  we set
$$
\SB^{E,\le z}=\SB^{\Omega,<z}/\alpha\SB^{\Omega,<z}.
$$
Note that there is no typo here. We indeed use $\SB^{\Omega,<z}$ to construct the left hand side. The map $\rho^{<z}_{\Omega, E}\colon\SB^{\Omega,<z}\to \SB^{E,<z}$ factors over the quotient $\SB^{\Omega,<z}\to\SB^{\Omega,<z}/\alpha\SB^{\Omega,<z}$ and in this way we obtain the restriction homomorphism  
$$
\SB^{E,\le z}=\SB^{\Omega,<z}/\alpha\SB^{\Omega,<z}\to\SB^{E,<z}.
$$

Now we construct the stalks $\SB^{\Omega,\le z}$. Let us first suppose that $z\not\in\Omega$, then 
$$
\SB^{\Omega,\le z}=\SB^{\Omega,<z}
$$ 
and the restriction map is the identity.
This  guarantees
$$
\supp_{\le}\,\SB^\Omega\subset\Omega.
$$
Let $E\colon \Omega\stackrel{\alpha}\llinie \Omega^\prime$ be an edge. If $z\not\in\Omega^\prime$ we have $\SB^{E,\le z}=\SB^{E,<z}$, so we define
$$
\rho_{\Omega,E}^{\le z}=\rho_{\Omega,E}^{<z}.
$$
If $z\in\Omega^\prime$, then 
$\SB^{E,\le z}=\SB^{\Omega,<z}/\alpha\SB^{\Omega,<z}$
and we let $\rho_{\Omega,E}^{<z}$ be the canonical homomorphism $\SB^{\Omega,<z}\to \SB^{\Omega,<z}/\alpha\SB^{\Omega,<z}$.

Now we deal with the case $z\in\Omega$. We can already define the local sections
$$
\Gamma(<z,\SB)=\left\{(m_\Theta)\in\bigoplus_{\Theta}\SB^{\Theta,<z}\left|\begin{matrix}\rho_{\Theta,E}^{<z}(m_\Theta)= \rho_{\Theta^\prime,E}^{<z} (m_{\Theta^\prime})\\ \text{ for all edges $E\colon\Theta\llinie\Theta^\prime$}\end{matrix}\right\}\right..
$$
Let $E\colon\Theta\llinie\Omega$ be an edge. Then $z\not\in\Theta$, hence $\SB^{\Theta,<z}=\SB^{\Theta,\le z}$ and we have already constructed the homomorphism $\rho_{E,\Theta}^{\le z}\colon \SB^{\Theta,\le z}\to\SB^{E,\le z}$. Hence we can already define $\SB^{\delta z}$.
 We define $\SB^{\Omega,\le z}$ as the projective cover of $\SB^{\delta z}$ in the category of graded $S$-modules. It comes with a homomorphism $\SB^{\Omega,\le z}\to\SB^{\delta z}$. The latter space  is contained inside $\left(\bigoplus_{E\colon \Theta\llinie\Omega}\SB^{E,<z}\right)\oplus \SB^{\Omega,<z}$, and the projection onto the $E$-component yields the  homomorphism
$$
\rho_{\Omega,E}^{\le z}\colon\SB^{\Omega,\le z}\to\SB^{E,\le z},
$$
which we take as the restriction map. 
\end{proof}

\subsection{Properties of $\SB(w)$}
Let $w\in\CV$. In this section we collect some properties of the sheaf $\SB(w)$.\begin{proposition} \label{prop-propBw} 
\begin{enumerate}
\item The sheaf $\SB(w)$ is flabby.
\item For each vertex $\Omega$ and all $z\in\Omega$, the homomorphism $\Gamma(\le z,\SB)\to\SB^{\Omega,\le z}$ is surjective.
\end{enumerate}
\end{proposition}

\begin{proof} 
Note that by construction, for each vertex $\Omega$ and each $z\in\Omega$ the image of  $u_z\colon\Gamma(>z,\SB)\to\SB^{\delta z}$ is contained in the image of $d_z\colon\SB^{\Omega,\le z}\to\SB^{\delta z}$. This implies (1) by Lemma \ref{lemma-flab}. 

So let us now show property (2). For $z\in\CV$ consider the cartesian diagram

\centerline{
\xymatrix{
\Gamma(\le z,\SB)\ar[d] \ar[r] & \SB^{\Omega,\le z} \ar[d]\\
\Gamma(<z,\SB) \ar[r] & \SB^{\delta z}. 
}
}
\noindent
Note that the lower horizontal map is surjective by definition of $\SB^{\delta z}$, and the vertical map on the left hand side is surjective by what we have already proven. So the composition $\Gamma(\le z,\SB)\to \SB^{\Omega,\le z}\to\SB^{\delta z}$ is surjective. As  $\SB^{\Omega,\le z}\to\SB^{\delta z}$ is a projective cover, $\Gamma(\le z,\SB)\to\SB^{\Omega,\le z}$ is surjective.
 \end{proof}
 \subsection{The global sections of $\SB(w)$}
 
 Let us denote by $B(w)$ the $(\CV,\le)$-cofiltered $\lCZ$-module obtained from $\SB(w)$ by taking sections.
\begin{lemma} \begin{enumerate}
\item The object $B(w)$ is contained in $\CC^{\refl}$.
\item There is an admissible epimorphism $\pi\colon B(w)\to\Delta(w)$.
\end{enumerate}
\end{lemma}

\begin{proof} We have already shown that $B(w)$ is flabby. The support condition (S) is an immediate consequence of the fact that $B(w)$ is constructed from the local sections of a cofiltered sheaf.  It remains to prove that $B(w)^{\CJ}=\Gamma(\CJ,\SB(w))$ is a reflexive $S$-module for all open susbets $\CJ$ of $\CV$. But this follows from Lemma \ref{lemma-refl} and the fact that $\SB(w)^{\Omega,\le x}$ is a reflexive $S$-module for any vertex $\Omega$ and $\SB(w)^{E,\le x}$ is a reflexive $S/\alpha S$-module for any edge labeled by $\alpha$. (Note that $\SB(w)^{\Omega,\le x}\subset\prod_{y\in\Omega,y\le x}\SB(w)^{\Omega,\le y}$ is a reflexive submodule.)
\end{proof}

In the next section we prove that $B(w)$ is a projective object in $\CC^{\refl}$. For this we need the following results.

\begin{lemma} \label{lemma-Bdeltax}  Let $\Omega$ be a vertex of $\lCG$ and $x\in\Omega$.
\begin{enumerate}
\item We have $B(w)^{\delta x}=\SB(w)^{\delta x}$.
\item The kernel of the homomorphism $B(w)^{<x}\to B(w)^{\delta x}$ is 
$$
\{(b_\Gamma)\in B(w)^{<x}\mid b_{\Omega}=0 \text{ and } b_\Gamma\in \alpha B(w)^{\Gamma,<x} \text{ if $\Gamma\stackrel{\alpha}\llinie\Omega$ }\}.
$$
\end{enumerate}
\end{lemma}
\begin{proof} For convenience we set $B=B(w)$ and $\SB=\SB(w)$.  By definition, $B^{\delta x}=(B^{\le x})^\Omega/B_{[x]}$. By Proposition \ref{prop-propBw}, $(B^{\le x})^\Omega=(\SB^{\Omega})^{\le x}$. Moreover, we can identify $B_{[x]}$ with  the space of sections in $\Gamma(\le x,\SB)$ which are supported on $\{x\}$. This is the set of all elements $m$ in $(\SB^{\Omega})^{\le x}$ with $m|_{<x}=0$ and $\rho_{\Omega^\prime,E}^{\le x}(m)=0$. By construction, this coincides with the kernel of the projective cover map $d_x\colon \SB(w)^{\Omega,\le x}\to \SB(w)^{\delta x}$. As this latter map is surjective, we obtain an isomorphism
$$
B(w)^{\delta x}=(B(w)^{\le x})^\Omega/B(w)_{[x]}=\SB^{\Omega,\le x}/\ker d_x=\SB(w)^{\delta x}.
$$
This shows claim (1).

For claim (2), we can identify $B(w)^{<x}$ with $\Gamma(<x,\SB(w))$ and, using (1), the kernel of $B(w)^{<x}\to B(w)^{\delta x}$ with the kernel of $\Gamma(<x,\SB(w))\to\SB(w)^{\delta x}$. By construction, this is what is stated above. 
\end{proof}

\subsection{The projectivity of $B(w)$}

Here we prove the main result of this article, namely the projectivity of the cofiltered objects $B(w)$ in the exact category $\CC^{\refl}$. We show, moreoever, that each indecomposable projective object occurs, up to a shift in the grading, in this way.

\begin{theorem}\begin{enumerate}
\item  For each $w\in\CV$, the object $B(w)$ is projective in $\CC^{\refl}$. 
\item Let $P$ be an indecomposable projective object in  $\CC^{\refl}$. Then there are  uniquely determined   $w\in \CV$ and $n\in \DZ$ such that $ P$ is isomorphic to $B(w)[n]$.
\end{enumerate}
\end{theorem}
\begin{proof} We prove part (1). For this we  check the two properties listed in Proposition \ref{prop-projobj}. For each $\Omega$ and all $x\in\Omega$,  $B(w)^{\le x,\Omega}$ is a projective graded $S$-module by construction. So it remains to check that for each $M\in\CC^{\refl}$ and all $x\in\CV$ the homomorphism $\Hom(B(w),M^{\le x})\to \Hom(B(w),M^{<x})$ is surjective. So suppose we are given a homomorphism $f\colon B(w)^{<x}\to M^{<x}$. We claim that there is a (unique) homomorphism $f^{\delta x}\colon B(w)^{\delta x}\to M^{\delta x}$ such that

\centerline{
\xymatrix{
M^{<x} \ar[r] & M^{\delta x} \\
B(w)^{<x}\ar[r]\ar[u]& B(w)^{\delta x}\ar[u] 
}
}
\noindent
commutes. By Lemma \ref{lemma-Bdeltax}  it suffices to show that an element $(m_\Gamma)\in M^{<x}$ with $m_\Omega=0$ and $m_\Gamma\in\alpha M^{<x,\Gamma}$ for each edge $\Gamma\stackrel{\alpha}\linie\Omega$  maps to zero in $M^{\delta x}$. But this is Lemma \ref{lemma-locext}. 

Now we deal with the classification. Let $w$ be a minimal vertex in $\CV$ with the property that $P^{\leq w}\neq 0$. By the minimality of $w$ we have that $P^{\le w}$ is isomorphic to a direct sum of degree-shifted copies of $\Delta(w)$. Let $\Delta(w)[n]$ be one of its direct summands. Then there is an admissible surjection $h\colon P\to \Delta(w)[n]$. Now, by projectivity we have homomorphisms $f\colon B(w)[n]\to P$ and $g\colon P\to B(w)[n]$ such that the diagram 

\centerline{
\xymatrix{
B(w)[n]\ar[r]^f\ar[dr] &P\ar[r]^g\ar[d]^h& B(w)[n]\ar[dl]\\
&\Delta(w)[n]&
}
}
\noindent
commutes.  The composition $g\circ f$ is hence an endomorphism of $B(w)[n]$ that is not nilpotent, hence must be an automorphism by Theorem \ref{thm-KrullSchm}. As $P$ is indecomposable, this implies that $B(w)[n]\cong P$. Finally, the uniqueness of $w$ and $n$ follows from the fact that $w$ is the smallest element in $\supp_{\le}\, B(w)[n]$ and $B(w)^{\le w}[n]\cong \Delta(w)[n]\cong S[n]$ as  graded $S$-modules.

\end{proof}

\end{document}